\documentclass[10pt]{article}
\usepackage[utf8]{inputenc}
\usepackage{amsmath}
\usepackage{amssymb}
\usepackage{mathrsfs}
\usepackage{amsthm}
\usepackage{mathabx}
\usepackage{tikz-cd}
\usepackage{mathtools}
\usepackage{dsfont}
\usepackage[all]{xy}
\usepackage{stmaryrd}

\newcommand{\qhyp}[5]{\,\mbox{}_{#1}\phi_{#2}\!\left(
  \genfrac{}{}{0pt}{}{#3}{#4};#5\right)}

\newcommand{\id}{\mathrm{id}}
\newcommand\restr[2]{{% we make the whole thing an ordinary symbol
  \left.\kern-\nulldelimiterspace % automatically resize the bar with \right
  #1 % the function
  \littletaller % pretend it's a little taller at normal size
  \right|_{#2} % this is the delimiter
  }}

\newcommand{\littletaller}{\mathchoice{\vphantom{\big|}}{}{}{}}
\theoremstyle{plain}
\newtheorem{thm}{Theorem}[section]
\newtheorem{lem}[thm]{Lemma}
\newtheorem{prop}[thm]{Proposition}

\theoremstyle{definition}

\newtheorem{defn}[thm]{Definition}

\theoremstyle{remark}
\newtheorem{Rem}[thm]{Remark}

\theoremstyle{plain}
\newtheorem*{thm*}{Theorem}
\newtheorem*{lem*}{Lemma}
\newtheorem*{prop*}{Proposition}
\newtheorem*{cor*}{Corollary}
\newtheorem*{conj*}{Conjecture}

\theoremstyle{definition}
\newtheorem*{ass*}{Assumption}
\newtheorem*{defn*}{Definition}

\theoremstyle{remark}
\newtheorem*{Rem*}{Remark}
\newcommand{\compconj}[1]{%
  \overline{#1}%
}
\makeatletter
\newcommand*{\rom}[1]{\expandafter\@slowromancap\romannumeral #1@}
\makeatother
\DeclareMathOperator{\End}{End}

\DeclareMathOperator{\Hom}{Hom}
\DeclareMathOperator{\spn}{span}
\DeclareMathOperator{\Ima}{Im}

\DeclareMathOperator{\supp}{supp}

\title{A proper $1$-cocycle on a Podle\'{s} sphere}
\author{Masato Tanaka
\footnote{%
tanakamasato.2121@gmail.com, masato.tanaka.c7@math.nagoya-u.ac.jp%
}}
\date{November 2022}

\begin{document}

\maketitle
\begin{abstract}
We define $1$-cocycles on coideal $*$-subalgebras of CQG Hopf $\ast$-algebras and consider the condition for $1$-cocycles to extend to $1$-cocycles on Drinfeld double coideals. We construct a $1$-cocycle on a Podle\'{s} sphere, which extends to a $1$-cocycle on $SL_q(2,\mathbb{R})$ and prove that it is proper.
\end{abstract}
\section{Introduction}
%The group $ SL(2,\mathbb{R})$ has been studied in years in great depth by a lot of mathematicians and appears in many places. For example, the group $SL(2,\mathbb{R})$ appears in the theory of automorphic forms, physics and representation theory etc. (cf.\ \cite{Lang})%
Let us review a brief history of the representation theory of the special linear group $SL(2,\mathbb{R})$. The fundamental works were worked out by I.\ Gelfand and M.\ Naimark \cite{GN} in 1946, V.\ Bargmann \cite{Barg} in 1947 and Harish-Chandra \cite{H-C} in 1952. In particular, all the unitary irreducible representations of $SL(2,\mathbb{R})$ were classified. By their works, we can understand the unitary dual of the special linear group $SL(2,\mathbb{R})$. For the concise account of representation theory and the unitary dual of $SL(2,\mathbb{R})$, see \cite[Chapter \rom{6}, Section 6]{Lang} .\par
%kokoninanikakakitai.
In 2021, K.\ De Commer and J.\ R.\ Dzokou Talla constructed a quatization of the group $SL(2,\mathbb{R})$ as the Drinfeld double coideal of the Drinfeld double $\mathcal{O}_q(SU(2))\bowtie U_q(\mathfrak{su}(2))$ \cite{DCDz1}. To construct $SL_q(2,\mathbb{R})$ (the quantum $SL(2,\mathbb{R})$) they constructed $U_q(\mathfrak{sl}(2,\mathbb{R})_t)$ which is a deformation of the enveloping algebra of the special linear group $SL(2,\mathbb{R})$. They classified the irreducible unitary representations of  $U_q(\mathfrak{sl}(2,\mathbb{R})_t)$. The unitary dual of $SL_q(2,\mathbb{R})$ is almost the same as the unitary dual of $SL(2,\mathbb{R})$. If one has the precise information on the unitary representations, then one can understand many properties of (quantum) groups. For example, one can construct $1$-cocycle and understand many properties such as the property (T), the Haagerup property et cetera via the representation theory. In the present paper we consider $1$-cocycles.\par
For the case of groups, (concrete constructions of) $1$-cocycles have played important roles in many areas such as L\'{e}vy process on groups \cite{FGT}, \cite{Mic}. One can characterize the property (T) and the Haagerup property et cetera via $1$-cocycles \cite{BHV}. There are many other problems related to $1$-cocycles \cite{FGT}, \cite{Nishi}, \cite{pet}, \cite{pop}, \cite{maximal}. For the case of quantum groups, $1$-cocycles are considered in \cite{DFK}, \cite{DFSW}, \cite{Hunt1}, \cite{Mic}, \cite{SS}, \cite{Hunt2} et cetera. We construct a $1$-cocycle on $SL_q(2,\mathbb{R})$ based on \cite{DCDz1} and \cite{DCDz2} and compute its growth. Our main theorems are the following:
\begin{enumerate}
\item We construct a nontrivial Yetter--Drinfeld 1-cocycle $C$ of the Podle\'{s} sphere $\mathcal{O}_q(S_t^2)$ on the discrete series representation $\mathcal{D}_2^+ \oplus \mathcal{D}_2^-$. 
\item The growth of the 1-cocycle $C$ is given by
\[C^{n*} C^n = \dfrac{q+q^{-1}+(q+q^{-1})^{-1}\llbracket a \rrbracket^2\{a+2\}}{\{a\}^2\{a+1\}}
\operatorname{Re}(Q'_n ((q+q^{-1})/2))
,\]
where $Q_n$ is a one-parameter specialization of the Askey--Wilson polynomial $p_n$. Here,
\begin{align*}
    Q_n(\bullet)=p_n((q+q^{-1})/2)^{-1}p_n(\bullet),
    \end{align*}
    and $p_n$ is defined by
    \begin{align*}
   & p_n(\cos\theta; -q^{-2a+1}, -q^{2a+1}, q, q| q^2)\\
    &=(-1)^{n}q^{n(-2a+1)}(q^2,-q^{-2a+2},-q^{-2a+2};q)_n\\
    &\cdot\qhyp 43{q^{-n},q^{n+3},-q^{-2a+1}e^{\sqrt{-1}\theta},-q^{-2a+1}e^{-\sqrt{-1}\theta}}{q^2,-q^{-2a+2},-q^{-2a+2}}{q,q}.
\end{align*}
\item In particular $C$ is proper, meaning that $C^{n*}C^n\to\infty$.
% \item Any $1$-cocycle satisfying the Yetter-Drinfeld condition extends to a $1$-cocycle on the Drinfeld double coideal (Theorem \ref{mainext}).
%\item We can construct a non-trivial $1$-cocycle $C$ on the Podle\'{s} sphere (Theorem \ref{construction1}). 
%\item We can extend $C$ on the Drinfeld double coideal $\mathcal{O}_q(S_t^2)\bowtie\mathcal{I}_c$ (Theorem \ref{construction2}).
%\item The representation to which our $1$-cocycles are associated is $\mathcal{D}_2^+\oplus\mathcal{D}_2^-$ (Theorem \ref{construction1},Theorem \ref{construction2}).
%\item We can compute the growth of our $1$-cocycle $C$ and show that $C$ is proper (Theorem \ref{growth}, Theorem \ref{proper}).
\end{enumerate}
Acknowledgements: The author wishes to express his thanks to his supervisor Prof. Y. Arano for many helpful suggestions. He is greatly indebted to Prof. S. Yanagida for pointing out mistakes in the first draft of the paper. He also gratefully acknowledges many encouragement of Prof. Y. Ueda during the preparation of the paper. Without their advice and warm encouragement, this thesis could not be written. This work is supported by the “Nagoya University 
Interdisciplinary Frontier Fellowship” supported by Nagoya University and JST, the establishment of university fellowships towards the creation of science technology innovation, 
Grant Number JPMJFS2120.

\section{Preliminaries}  The best general references here are \cite{DCDz1}, \cite{DCDz2}, \cite{NT}, and \cite{KS}. 
\subsection{Notations and conventions}
\begin{itemize}
\item The vector spaces and algebras which we consider are the complex vector spaces unless otherwise stated.
\item A preHilbert space is written in a calligraphic font. Its completion is written in a block letter. For example we denote the completion of a preHilbert space $\mathcal{H}$ by $H$ 
\item The inner products are conjugate linear in the first variable and linear in the second variable.
\item Let $\mathcal{H}$ be a preHilbert space. Let  $\End({\mathcal{H}})$ denote the space of all adjointable linear maps on $\mathcal{H}$.
\item Let $H$ be a Hilbert space. We define $\mathbb{B}(H)$ to be the set of bounded operators on $H$.
\item Let $\mathcal{H}$ be a preHilbert space. For $\xi,\eta\in\mathcal{H}$ we define a map $\xi\eta^*$ by  $\xi\eta^*(\zeta)=\langle\eta,\zeta\rangle\xi\quad(\zeta\in\mathcal{H})$
\item We denote the spectrum of an element $a$ in a Banach algebra by $\sigma(a)$.
\item In this thesis $q$ is a real number in the interval $(0,1)$.
\item For each $a\in\mathbb{R}$, we put $[a]\coloneqq\dfrac{q^a-q^{-a}}{q-q^{-1}}$, $\llbracket a \rrbracket \coloneqq q^a-q^{-a}$, $\{a\}\coloneqq q^a+q^{-a}$.
\item We use the $q$-Pochhammer symbols:
\begin{align*}
 (b;q)_\infty &\coloneqq\prod_{i=0}^\infty (1-b q^i),\\
 (b;q)_n&\coloneqq(b;q)_\infty / (b q^n;q)_\infty ,\\
(b_1,b_2,\ldots,b_r;q)_n&\coloneqq(b_1;q)_n(b_2;q)_n\cdots(b_r;q)_n.
\end{align*}
\item We also use $q$-hypergeometric series:
\begin{align*}
   \qhyp {s+1}s{a_1,\ldots,a_{s+1}}{b_1,\ldots, b_s}{q,z}\coloneqq\displaystyle\sum_{i=0}^{\infty}\dfrac{(a_1,a_2,\ldots,a_{s+1};q)_iz^i}{(b_1,b_2,\ldots,b_s;q)_i(q;q)_i}_{.}
\end{align*}
\end{itemize}
\subsection{Drinfeld doubles}
In this subsection we define unitary pairings, Drinfeld doubles and Drinfeld double coideals.
\begin{defn}
Let $A$ and $U$ be Hppf $\ast$-algebras. We call the linear map $(-,-) \colon A\otimes U \to \mathbb{C} $ a \emph{unitary pairing} if the following conditions hold: $(\Delta(a), h\otimes k)=(a,hk)$, $(a\otimes b, \Delta(h))$, $(1,h)=\epsilon(h)$, $(a,1)=\epsilon(a)$, $(a^\ast, h)=\compconj{(a, S(h)^\ast)}$ and $(a, h^\ast)=\compconj{(S(a)^\ast, h)}$ for all $a,b\in A$ and $h,k\in U$. We define the operations $\lhd$ and $\rhd$ as follows.
\begin{align*}
h\rhd a&\coloneqq a_{(1)}(a_{(2)}, h)=(\id\otimes (\bullet, h))\Delta(a)\\
a\lhd h&\coloneqq (a_{(1)}, h)a_{(2)}=((\bullet, h)\otimes\id)\Delta(a)\\
a\rhd h&\coloneqq h_{(1)}(a, h_{(2)})=(\id\otimes (a,\bullet))\Delta(h)\\
h\lhd a&\coloneqq (a, h_{(1)})h_{(2)}=((a, \bullet)\otimes\id)\Delta(h)
\end{align*}
\end{defn}
\begin{defn}
Let $A$ and $U$ be Hopf $\ast$-algebras with a unitary pairing. We define the \emph{Drinfeld double} $A\bowtie U$ of $A$ and $U$ to be the $\ast$-algebra generated by $A$ and $U$ with the commuting relations
\begin{align*}
&ha=a_{(2)}(S^{-1}(a_{(1)})\rhd h \lhd a_{(3)})=(a_{(3)},  h_{(1)})a_{(2)}h_{(2)}(a_{(1)},  S^{-1}(h_{(3)}))\\
&ah=h_{(2)}(S^{-1}(h_{(1)})\rhd a \lhd h_{(3)})=(S^{-1}(a_{(3)}), h_{(1)})h_{(2)}a_{(2)}(a_{(1)}, h_{(3)})
\end{align*}
for each $a\in A$ and $h\in U$.
\end{defn}
\begin{defn}
Let $A$ and $U$ be as above. Let $I\subset U$ be a unital right coideal $\ast$-subalgebra and let $J\subset I^{\perp}\subset A$ be a unital left coideal $\ast$-subalgebra, where $I^{\perp}\coloneqq\{a\in A\mid h\rhd a=\epsilon(h)a\hspace{1mm}\text{for all}\hspace{1mm}h\in I\}$. Then we define the \emph{Drinfeld double coideal} $J\bowtie I$ of $I$ and $J$ to be the $\ast$-subalgebra of $A\bowtie U$ spanned by the elements $by$, where $b\in J$ and $y\in I$. This is a unital left coideal $\ast$-subalgebra of $A\bowtie U$.
\end{defn}
\begin{Rem}
Notice that the commuting relations in the Drinfeld double coideal $J\bowtie I$ can be simplified:
\begin{align*}
 &by=y_{(1)}(b\lhd y_{(2)}), yb=(b\lhd S^{-1}(y_{(2)}))y_{(1)}\quad(b\in J, y\in I)
 \end{align*}
\end{Rem}
\subsection{The Hopf $\ast$-algebras $\mathcal{O}_q(SU(2))$ and $U_q(\mathfrak{su}(2))$}\label{exmpofhopf}
In this subsection, we introduce the Hopf $\ast$-algebras {$\mathcal{O}_q(SU(2))$ and $U_q(\mathfrak{su}(2))$. We omit the proofs of all the statements in the present subsection. See \cite{NT} and \cite{KS} for more details.
\begin{defn}The algebra
$\mathcal{O}_q(SU(2))$ is the unital associative algebra generated by the elements $\alpha, \gamma$ with the relations\par
$\alpha^{\ast}\alpha+\gamma^{\ast}\gamma=1, \alpha\alpha^{\ast}+q^2\gamma\gamma^{\ast}=1, \gamma^{\ast}\gamma=\gamma\gamma^{\ast}, \alpha\gamma=q\gamma\alpha, \alpha\gamma^{\ast}=q\gamma^{\ast}\alpha$.\\
This relations means that $U_{1/2}=[u_{i,j}]_{i,j=1/2,-1/2}\coloneqq\begin{bmatrix}
\alpha & -q\gamma^{\ast} \\
\gamma & \alpha^{\ast} \\
\end{bmatrix}$ is unitary.\\
The coproduct is defined by $\Delta(u_{i,j})=\sum_{k}u_{i,k}\otimes u_{k,j}$.\\
We define the counit $\epsilon$ and the antipode $S$ by the formulas $\epsilon(\alpha)=1$, $\epsilon(\gamma)=0$, $S(\alpha)=\alpha^{\ast}, S(\alpha^{\ast})=\alpha, S(\gamma)=-q\gamma, S(\gamma^{\ast})=-q^{-1}\gamma^{\ast}$.
Note that $\epsilon$ and $S$ extends to $\mathcal{O}_{q}(SU(2))$.
\end{defn}

Next, we consider the algebra $U_q(\mathfrak{su}(2))$, which is \emph{dual} to $\mathcal{O}_q(SU(2))$ as a Hopf $\ast$-algebra.
\begin{defn}
We define $U_q(\mathfrak{su}(2))$ to be the unital associative algebra generated by the elements $k, k^{-1}, e$ and $f$ with the relations
\begin{align*}
    kk^{-1}=k^{-1}k=1,\quad ke=q^2ek,\quad kf=q^{-2}fk,\quad [e,f]=\dfrac{k-k^{-1}}{q-q^{-1}}.
\end{align*}
The Hopf $\ast$-algebra structure is given by the following:
\begin{align*}
 &\Delta(k^{\pm})=k^{\pm}\otimes k^{\pm},\quad\Delta(e)=e\otimes 1+k\otimes e,\quad\Delta(f)=1\otimes f+f\otimes k^{-1}\\
 %\end{align*}
 %\begin{align*}
    &\qquad\qquad\qquad\qquad\epsilon(k^{\pm})=1,\quad\epsilon(e)=0,\quad\epsilon(f)=0\\
    &\qquad\quad\quad\quad S(k^{\pm})=k^{\mp},\quad S(e)=-k^{-1}e,\quad S(f)=-fk\\
    &\qquad\qquad\qquad\qquad
    k^*=k,\quad e^*=fk,\quad f^*=k^{-1}e
 \end{align*}
 
\end{defn}
For each non-negative half-integer $s\in\dfrac{1}{2}\mathbb{Z}_{\geq 0}$, we take the $(2s+1)$-dimensional irreducible corepresentation $(H_s, U_s)$ of $SU_q(2)$ and fix the orthonormal basis $(\xi_{i}^s)_{i=-s,\ldots s}$, where the vector $\xi_i^s$ is of weight $s$. We denote each matrix coefficient of the corepresentation $U_s$ by $u^s_{\xi,\eta}\quad(\xi,\eta\in H_s)$. We denote the matrix coefficients $u_{\xi_i^s, \xi_j^s}^s$ by $u^s_{ij}$ for short. \par
We note that if $x\in U_q(\mathfrak{su}(2))$ is given, then we have $x\rhd u^s_{\xi,\eta}=u^s_{\xi,x\eta}$ and $u^s_{\xi, \eta}\lhd x=u^s_{x^{^{\ast}}\xi, \eta}$.
%We define the map $R\colon U_q(\mathfrak{su}(2))\rightarrow&U_q(\mathfrak{su}(2))$ by requiring that
%\begin{align*}
 %   R(k)=k^{-1},\quad R(e)=-qk^{-1}e,\quad R(f)=-q^{-1}fk.
%\end{align*}
%We may write this map formally as $R(\bullet)\coloneqq k^{1/2}S(\bullet)k^{-1/2}$. This is a $\ast$-preserving anti-multiplicative linear map. Note that $\Delta\circ R=(R\otimes R)\Delta^{\op}$.
%\begin{prop}[Representations parametrized by half-integers]
%For each $s\in\tfrac{1}{2}\mathbb{Z}_{\geq0}$, there exist a (2s+1)-dimensional Hilbert space, an orthonormal basis $(\xi_i^s)_{i=-s,\ldots,s}\subset H_s$ and a unitary $\ast$-representation $\pi_s\colon U_q(\mathfrak{su}(2))\rightarrow&\mathbb{B}(H_s)$ such that 
%\begin{align*}
 %   \pi_s(k)\xi_i^s&=q^{2i}\xi_i^s\\
  %  \pi_s(e)\xi_i^s&=q^{i+1}([s-i][s+i+1])^{1/2}\xi_{i+1}^s\\
  %  \pi_s(f)\xi_i^s&=q^{-i}([s+i][s-i+1])^{1/2}\xi^s_{i-1}.
%\end{align*}
%\end{prop}
Recall that the unitary pairing between the Hopf $\ast$-algebras $\mathcal{O}_q(SU(2))$ and $U_q(\mathfrak{su}(2))$ is given by
\begin{align*}
    (U_{1/2}, k)&=\begin{bmatrix}
q & 0 \\
0 & q^{-1} \\
\end{bmatrix},\\
(U_{1/2}, e)&=\begin{bmatrix}
0 & q^{1/2} \\
0 & 0 \\
\end{bmatrix},\\
(U_{1/2}, f)&=\begin{bmatrix}
0 & 0 \\
q^{-1/2} & 0 \\
\end{bmatrix}_{.}
\end{align*}
By this pairing, we consider that $\mathcal{O}_q(SU(2))$ is dual to $U_q(\mathfrak{su}(2))$ and vice versa.
\subsection{The Drinfeld double $U_q(\mathfrak{sl}(2,\mathbb{R})_t)$ and its representation theory}\label{rep}
Let $a\in\mathbb{R}$ and put $t=\llbracket a \rrbracket$. Define the elements $B_t, \Tilde{B}_t$ in $U_q(\mathfrak{su}(2))$ by
\begin{align*}
    B_t&\coloneqq -q^{1/2}k^{-1}e+q^{1/2}f-\sqrt{-1}(q-q^{-1})^{-1}tk^{-1}\\
    \Tilde{B}_t&\coloneqq B_t+\sqrt{-1}(q-q^{-1})^{-1}t.
\end{align*}
By \cite[Theorem 2.1]{DCDz1}\cite[Theorem 4.3]{K}, for each half-integer $s\in\tfrac{1}{2}\mathbb{Z}_{\geq0}$, we have 
\begin{align*}
\sigma(\pi_s(\sqrt{-1}B_t))
=\{[a+2i]\mid i=-s,-s+1,\ldots, s\},
\end{align*}
and $\pi_s(\sqrt{-1}B_t)$ is a diagonal operator with respect some orthonormal basis $\{\eta^s_{[a+2i]}\}_{i=-s,\ldots,s}$ for $H_s$ consisting of eigenvectors, where $\pi_s(\sqrt{-1}B_t)\eta^s_{[a+2i]}=[a+2i]\eta^s_{[a+2i]}$. We fix an orthonormal basis $\{\eta^s_{[a+2i]}\}_{i=-s,\ldots,s}$.
Put $I\coloneqq\mathbb{C}[B_t]$ and $B\coloneqq I^{\perp}$. We call $B$ a \emph{Podle\'{s} sphere} and denote it by $\mathcal{O}_q(S_t^2)$. $I$ is a right coideal and $\mathcal{O}_q(S_t^2)$ is a left coideal. We define $U_q(\mathfrak{sl}(2,\mathbb{R})_t)\coloneqq \mathcal{O}_q(S_t^2)\bowtie I$. This is a unital left coideal $\ast$-subalgebra of the Drinfeld double of $\mathcal{O}_q(SU(2))$.
\begin{Rem}
 This $B_t$ is different from the element $B_t$ in \cite{DCDz1}. Applying the unitary antipode to $B_t$ in \cite{DCDz1}, we get the element $B_t$ in the present thesis. Since we define the growth of the cocycle $C$ by $C^*C$, the coideal needs to be a right coideal. This is why we changed the notation.
\end{Rem}
Let us recall the representation theory of $U_q(\mathfrak{sl}(2,\mathbb{R})_t)$ as in \cite{DCDz1}.
\begin{defn}\cite[Definition 3.1]{DCDz1}
We say $U_q(\mathfrak{sl}(2,\mathbb{R})_t)$-module $(\mathcal{H},\pi)$ is an \emph{$SL(2,\mathbb{R})_t$-admissible $\ast$-representation} if the following three conditions are satisfied:
\begin{itemize}
    \item $\mathcal{H}$ is a preHilbert space
    \item $\pi(\sqrt{-1}B_t)$ is diagonalizable with finite dimensional eigenspaces and each eigenvalue of $\pi(\sqrt{-1}B_t)$ is of the form $[c]$, where $c\in a+\mathbb{Z}$
    \item We have $\langle\xi, \pi(x)\eta\rangle=\langle \pi(x^{\ast})\xi, \eta\rangle$
    for all $\xi,\eta\in\mathcal{H}$ and $x\in U_q(\mathfrak{sl}(2,\mathbb{R})_t)$.
\end{itemize}
\end{defn}
%By the remark before Proposition \ref{eigvals}, the information on eigenvalues for the element $\sqrt{-1}R(B_t)$ is the same as that of $\sqrt{-1}B_t$. The representation theory of our $U_q(\mathfrak{sl}(2,\mathbb{R})_t)$ is the same that of $U_q(\mathfrak{sl}(2,\mathbb{R})_t)$ in \cite{DCDz1}. \\

For each $c\in\mathbb{R}$ define the vectors in $H_1$ as follows:
\begin{align*}
    \xi^{(\llbracket c \rrbracket;1)}&\coloneqq q^{1/2}\xi^{1}_{1}-(\sqrt{-1}\llbracket c \rrbracket/\sqrt{q+q^{-1}})\xi^{1}_0+q^{-1/2}\xi^1_{-1},\\ 
    \xi_{+}^{(\llbracket c \rrbracket;1)}&\coloneqq -q^{-c-1/2}\xi^{1}_{1}+(\sqrt{-1}\sqrt{q+q^{-1}})\xi^{1}_0+q^{c+1/2}\xi^1_{-1},\\
    \xi_{-}^{(\llbracket c \rrbracket;1)}&\coloneqq  q^{c-1/2}\xi^{1}_{1}+(\sqrt{-1}\sqrt{q+q^{-1}})\xi^{1}_0-q^{-c+1/2}\xi^1_{-1}.
\end{align*}
Then the elements
\begin{align*}
    T^{+}_c&\coloneqq u^1_{\xi_{+}^{(\llbracket c \rrbracket;1)}, \xi^{(t;1)}}\\
    A_c&\coloneqq u^1_{\xi^{(\llbracket c \rrbracket;1)}, \xi^{(t;1)}}-(q^{-1}+q)^{-1}\llbracket c \rrbracket t\\
    T^{-}_c&\coloneqq u^1_{\xi_{-}^{(\llbracket c \rrbracket;1)}, \xi^{(t;1)}}
\end{align*}
satisfy the following two propositions(\cite[Proposition 3.4, Lemma 3.7]{DCDz1}):
\begin{prop}
Let $\mathcal{H}=(\mathcal{H}, \pi)$ be a $U_q(\mathfrak{sl}(2,\mathbb{R})_t)$-module and let $\mathcal{H}(c)\coloneqq\{\xi\in\mathcal{H}\mid\sqrt{-1}\pi(B_t)\xi=[c]\xi\}$. Then 
\begin{align*}
    T_c^+\mathcal{H}(c)&\subset\mathcal{H}(c+2)\\
    A_c\mathcal{H}(c)&\subset\mathcal{H}(c)\\
    T_c^-\mathcal{H}(c)&\subset\mathcal{H}(c-2)
\end{align*}
\end{prop}
\begin{prop} We have 
\begin{align*}
    T^-_{c+2}T^+_{c}&=(\{c-a+1\}-A_c)(\{c+a+1\}+A_c),\\
    T^+_{c-2}T_c^-&=(\{c-a-1\}-A_c)(\{c+a-1\}+A_c),\\
    A_{c+2}T^+_c&=T^+_{c}A_c\quad\text{and}\quad A_{c-2}T_c^-=T_c^-A_c.
    \end{align*}
\end{prop}
We define an important class of $U_q(\mathfrak{sl}(2,\mathbb{R})_t)$-modules.
\begin{defn}\cite[Definition 3.5]{DCDz1}Let $\lambda,b\in\mathbb{R}$.
We say a $U_q(\mathfrak{sl}(2,\mathbb{R})_t)$-module $(V,\pi)$ is \emph{$(\lambda,b)$-basic} if there exists $b\in\mathbb{R}$ and a cyclic vector $\xi\in V$ such that $\pi(\sqrt{-1}B_t)\xi=[b]\xi$ and $\pi(A_b)\xi=\lambda\xi$.
\end{defn}
By \cite[p. 11,12]{DCDz1}, one can construct the `largest' $(\lambda,b)$-basic modules $\mathcal{M}_{\lambda,b}$, meaning that each $\mathcal{M}_{\lambda,b}$ satisfies the following universal property(\cite[Proposition 3.11]{DCDz1}):
\begin{prop}\label{univ}Let $\lambda,b\in\mathbb{R}$. There exists a $(\lambda, b)$-basic module $\mathcal{M}_{\lambda, b}$ such that any $(\lambda,b)$-basic module is a quotient of $\mathcal{M}_{\lambda, b}$.
\end{prop}
Let $e_b\in\mathcal{M}_{\lambda,b}$ be a cyclic vector satisfying $\sqrt{-1}B_te_b=[b]e_b$. Put
\begin{align*}
  e_{b+2n}&\coloneqq T_b^{+,(n)}e_b\coloneqq T^+_{b+2n-2}T^+_{b+2n-4}\cdots T^+_{b}e_b\quad(n=1,2,\ldots),\\
  e_{b-2n}&\coloneqq T_b^{-,(n)}e_b\coloneqq T^-_{b-2n+2}T^-_{b-2n+4}\cdots T^-_{b}e_b\quad(n=1,2,\ldots).
\end{align*}
$\{e_{b+2n}\mid n\in\mathbb{Z}\}$ forms a basis for $\mathcal{M}_{\lambda,b}$ and $e_{b+2n}$ satisfies $\sqrt{-1}B_te_{b+2n}=[b+2n]e_{b+2n}$ for each $n\in\mathbb{Z}$. There exists a unique $\ast$-invariant sesquilinear form on $\mathcal{M}_{\lambda, b}$\cite[Proposition 3.12]{DCDz1}. Let $\mathcal{N}_{\lambda, b}$ denote the kernel of this sesquilinear form. Put $\mathcal{L}_{\lambda, b}=\mathcal{M}_{\lambda, b}/\mathcal{N}_{\lambda, b}$.\par
We have the list of $SL(2,\mathbb{R})_t$-admissible irreducible $\ast$-representations(\cite[Proposition 3.17]{DCDz1}). All the $SL(2,\mathbb{R})_t$-admissible irreducible $\ast$-representations we need for the construction of our $1$-cocycle are the following:
\begin{enumerate}
    \item The principal series representation $\mathcal{L}_{\lambda,a}=\mathcal{M}_{\lambda,a}$, whose cyclic vector is $e_a$ and $e_{a+2n}\neq0$ for all $n\in\mathbb{Z}$. ($0<\lambda<q+q^{-1}$)
    \item The discrete series representation $\mathcal{L}_{q+q^{-1},a+2}=\mathcal{D}_2^+$, whose cyclic vector is $e_{a+2}$. The vector $e_{a+2}$ satisfies $T^-_{a+2}e_{a+2}=0$ and $A_{a+2}e_{a+2}=(q+q^{-1})e_{a+2}$.
    \item The discrete series representation $\mathcal{L}_{q+q^{-1},a-2}=\mathcal{D}_2^-$, whose cyclic vector is $e_{a-2}$. The vector $e_{a-2}$ satisfies $T^-_{a-2}e_{a-2}=0$ and $A_{a-2}e_{a-2}=(q+q^{-1})e_{a-2}$.
    \item The trivial representation $\mathcal{L}_{q+q^{-1},a}=\mathds{1}$, whose cyclic vector is $e_a$. The vector $e_a$ satisfies $T^{\pm}_{a}e_a=0$ and $A_{a}e_{a}=(q+q^{-1})e_a$.
\end{enumerate}
Note that $\mathcal{M}_{\lambda,a}\to\mathds{1}$ as $\lambda\to q+q^{-1}$ in Fell's topology. Note also that each of $\mathcal{D}_{2}^{\pm}$ has an inner product $\langle\bullet,\bullet\rangle_{_{_{\mathcal{D}_{2}^{\pm}}}}$ such that $\langle T_{a}^{\pm}e_a,T_{a}^{\pm}e_a\rangle_{_{_{\mathcal{D}_{2}^{\pm}}}}=1$ respectively.
\subsection{The Drinfeld double coideal $\mathcal{O}_q(S^{2}_t)\bowtie\mathcal{I}_c$}
Let $\mathcal{A}$ be a CQG Hopf-$\ast$ algebra. Let $(U_{\alpha}\in\mathbb{B}(H_{\alpha})\otimes\mathcal{A})_{\alpha}$ be a maximal family of mutually inequivalent finite dimensional unitary corepresentations of $\mathcal{A}$, then as an algebra $\mathcal{A}=\bigoplus_{\alpha}\mathbb{B}(H_{\alpha})^{\ast}$. Put $\mathcal{U}\coloneqq\Hom(\mathcal{A},\mathbb{C})=\prod_{\alpha}\mathbb{B}(H_{\alpha})$ with the product topology. Put $\mathcal{U}\compconj{\otimes}\mathcal{U}\coloneqq\prod_{\alpha, \alpha'}\mathbb{B}(H_{\alpha})\otimes\mathbb{B}(H_{\alpha'})=\Hom(\mathcal{A}\otimes\mathcal{A},\mathbb{C})$. The product on $\mathcal{A}$ defines a unital $\ast$-homomorphism $\Delta\colon\mathcal{U}\to\mathcal{U}\compconj{\otimes}\mathcal{U}$. The coproduct on $\mathcal{A}$ defines a product on $\mathcal{U}$. There is a one to one correspondence between finite dimensional unitary corepresentations of $\mathcal{A}$ and finite dimensional continuous $\ast$-representations of $\mathcal{U}$. We denote the $\ast$-representation of $\mathcal{U}$ corresponding to corepresentation $U_{\alpha}$ by $\pi_{\alpha}\colon\mathcal{U}\rightarrow\mathbb{B}(H_{\alpha})$.
Let $\mathcal{B}\subset \mathcal{A}$ be a unital left coideal $\ast$-subalgebra of $\mathcal{A}$. Let $\mathcal{B}_{+}\coloneqq\mathcal{B}\cap\ker(\epsilon)$. Then $\mathcal{C}\coloneqq\mathcal{A}/\mathcal{B}_{+}\mathcal{A}$ has a structure of a right $\mathcal{A}$-module coalgebra such that the canonical projection $p_{\mathcal{C}}\colon\mathcal{A}\twoheadrightarrow\mathcal{C}$ is a homomorphism of right $\mathcal{A}$-module coalgebras \cite[Proposition1]{Tak}. Let $\mathcal{I}\coloneqq\Hom(\mathcal{C}, \mathbb{C})$. The surjection $p_{\mathcal{C}}$ induces an inclusion $\mathcal{I}\subset\mathcal{U}$. Let $(G_{\beta}, \theta_{\beta})_{\beta}$ be the maximal family of the mutually inequivalent $\ast$-representations of $\mathcal{I}$ such that for each $\beta$, the representation $\theta_{\beta}$ arises as a subrepresentation of $\pi_{\alpha}$ for some $\alpha$. Then $\mathcal{I}=\prod_{\beta}\mathbb{B}(G_{\beta})$. The right $\mathcal{A}$-module structure on $\mathcal{C}$ dualizes to a unital $\ast$-homomorphim $\Delta\colon\mathcal{I}\to\mathcal{I}\compconj{\otimes}\mathcal{U}=\prod_{\alpha,\beta}\mathbb{B}(G_{\beta})\otimes\mathbb{B}(H_{\alpha})$. As in \cite{DCDz2}, let $\Phi_{\mathcal{C}}\in\mathcal{I}$ be the self-adjoint support projection of $\epsilon\colon\mathcal{I}\to\mathbb{C}$. This projection plays an important role. We put $\mathcal{I}_c\coloneqq\bigoplus_{\beta}\mathbb{B}(G_{\beta})$, then $\Phi_{\mathcal{C}}\in\mathcal{I}_c$. Let $\mathcal{I}_{\infty}\coloneqq l_{\infty}$-$\prod_{\beta}\mathbb{B}(G_{\beta})$ be the set of uniformly bounded elements in $\mathcal{I}$. We call $\mathcal{I}_c\subset\mathcal{I}_{\infty}\subset\mathcal{I}$ the \emph{stabilizer duals} of $\mathcal{B}$. We denote the projection onto the direct summand $\mathbb{B}(G_{\beta})$ by $\Phi_{\beta}$.\par
If $\mathcal{A}=\mathcal{O}_q(SU(2))$ and $\mathcal{B}=\mathcal{O}_q(S_t^2)$, then $\mathcal{I}_c=\bigoplus_{n\in\mathbb{Z}}\mathbb{B}(\mathbb{C}_n)$, where $\mathbb{C}_n$ is the one-dimensional vector space on which $\sqrt{-1}B_t$ acts as $[a+2n]$. We denote the projection from $\mathcal{I}_c$ onto the $n$-th direct summand $\mathbb{B}(\mathbb{C}_n)$ by $\Phi_n$. Let $\mathcal{U}\coloneqq \Hom(\mathcal{A},\mathbb{C})$. Inside of the Drinfeld double $\mathcal{A}\bowtie\mathcal{U}$, we have a unital left coideal $\ast$-subalgebra $\mathcal{O}_q(S_t^2)\bowtie\mathcal{I}_c$ generated by $\mathcal{O}_q(S_t^2)$ and $\mathcal{I}_c$. The orthogonality relations for this subalgebra is as follows.
\begin{align*}
 &by=y_{(1)}(b\lhd y_{(2)}), yb=(b\lhd S^{-1}(y_{(2)}))y_{(1)}\quad(b\in \mathcal{O}_q(S_t^2), y\in \mathcal{I}_c).   
\end{align*}
One can show that the following correspondence between representations on a preHilbert space $\mathcal{H}$ of $U_q(\mathfrak{sl}(2,\mathbb{R})_t)$, $\mathcal{O}_q(S_t^2)\bowtie\mathcal{I}$ and $\mathcal{O}_q(S_t^2)\bowtie\mathcal{I}_c$.
\begin{enumerate}
    \item $SL(2,\mathbb{R})_t$-admissible $\ast$-representation of $U_q(\mathfrak{sl}(2,\mathbb{R})_t)$ on $\mathcal{H}$
    \item Unital continuous $\ast$-representation of $\mathcal{O}_q(S_t^2)\bowtie\mathcal{I}$ on $\mathcal{H}$, where $\Phi_n\mathcal{H}$ is finite dimensional for each $n$.
    \item Non-degenerate  $\ast$-representation of $\mathcal{O}_q(S_t^2)\bowtie\mathcal{I}_c$ on $\mathcal{H}$, where $\Phi_n\mathcal{H}$ is finite dimensional for each $n$.
\end{enumerate}
We show that any irreducible $\ast$-representation of $\mathcal{O}_q(S_t^2)\bowtie\mathcal{I}_c$ induces an $SL(2,\mathbb{R})_t$-admissible irreducible $\ast$-representation of $U_q(\mathfrak{sl}(2,\mathbb{R})_t)$. This is an easy corollary of the following lemma.
\begin{lem}
    Let $\pi\colon\mathcal{O}_q(S_t^2)\bowtie\mathcal{I}_c\to\mathbb{B}(H)$ be an irreducible $\ast$-representation on a Hilbert space $H$. Then the dimension of $\pi(\Phi_n)H$ is at most one.
\end{lem}
\begin{proof}
   Let $||\bullet||_u\coloneqq\sup\{||\theta(\bullet)||\mid\text{$\theta$ is a $\ast$-representation of $\mathcal{O}_q(S_t^2)\bowtie\mathcal{I}_c$}\}$, then $||\bullet||_u$ is a $C^{\ast}$-norm. Define $C_u(SL_q(2,\mathbb{R})_t)$ as the $C^*$-completion of the normed $\ast$-algebra $\mathcal{O}_q(S_t^2)\bowtie\mathcal{I}_c$ with respect to the norm $||\bullet||_u$. Extend $\pi$ to the $\ast$-representation of $C_u(SL_q(2,\mathbb{R})_t)$ and denote it by $\pi_u$.\par
   Since $\pi$ is irreducible, the extension $\pi_u$ is irreducible also and then 
   \begin{align*}
\pi_{u, n}\colon\Phi_nC_u(SL_q(2,\mathbb{R})_t)\Phi_n\to\mathbb{B}(\Phi_n H),\quad x\mapsto\Phi_n\pi(x)\Phi_n
   \end{align*}
   is irreducible or zero dimensional for each $n$. We claim that $\Phi_n(\mathcal{O}_q(S_t^2)\bowtie\mathcal{I}_c)\Phi_n$ is commutative. Indeed, we have $\mathcal{O}_q(S_t^2)=\spn\{T^{\pm,(k)}_{a+n}P(A_{a+n})\mid k=0,1,2,\ldots; P\hspace{1mm}\text{is a polynomial}\hspace{1mm} \}$ and 
\begin{align*}
\Phi_n T_{a+n}^{\pm,(k)} P(A_{a+n})\Phi_n= 
\begin{cases}
0 & \text{if}\hspace{1mm}k\neq0\\
P(A_{a+n})\Phi_n & \text{if}\hspace{1mm} k=0
\end{cases}
\end{align*}
This implies that $\Phi_n(\mathcal{O}_q(S_t^2)\bowtie\mathcal{I}_c)\Phi_n\subset\Phi_n\mathbb{C}[A_{a+n}]\Phi_n$. Thus the corner $\Phi_nC_u(SL_q(2,\mathbb{R})_t)\Phi_n$ is commutative
   and $\pi_{u, n}$ is a one dimensional or zero dimensional representation.
\end{proof}
\begin{thm}
   The $\ast$-representation of $U_q(\mathfrak{sl}(2,\mathbb{R})_t)$ corresponding to a given irreducible $\ast$-representation of $\mathcal{O}_q(S_t^2)\bowtie\mathcal{I}_c$ is an $SL(2,\mathbb{R})_t$-admissible irreducible $\ast$-representation. 
\end{thm}
\subsection{Askey--Wilson polynomials and matrix coefficients}\label{AW} We need the result in this subsection in order to compute the growth of our cocycle. The best general reference here is Koornwinder's paper \cite{K}. By making formal slight modifications of \cite[Lemma 4.6, Proposition 4.7, Remark 4.8]{K} we get the following theorem.
\begin{thm}\cite[Lemma 4.6, Proposition 4.7, Remark 4.8]{K}
% Let
 %  \[c_i^{n,-a}\coloneqq&\dfrac{(\sqrt{-1})^iq^{-(n-a)i}q^{i^2/2}}{(q^2;q^2)_{n+i}^{1/2}(q^2;q^2)_{n-i}^{1/2}}\qhyp 32{q^{-2n+2i},q^{-2n},-q^{-2n+2a}}{q^{-4n},0}{q^2,q^2}=c_{-i}^{n,-a}\]
  % and
%\[\alpha=&\displaystyle\sum_{i=-n}^nq^i|c_{i}^{n,-a}|^2.\]
 Let $Q_n$ be a one-parameter specialization of an Askey--Wilson polynomial:
\begin{align*}
Q_n(\bullet)=p_n((q+q^{-1})/2)^{-1}p_n(\bullet),
\end{align*} where the Askey--Wilson polynomial $p_n$ is defined by
    \begin{align*}
   & p_n(\cos\theta; -q^{-2a+1}, -q^{2a+1}, q, q| q^2)\\
    &=(-1)^{n}q^{n(-2a+1)}(q^2,-q^{-2a+2},-q^{-2a+2};q)_n\\
    &\cdot\qhyp 43{q^{-n},q^{n+3},-q^{-2a+1}e^{\sqrt{-1}\theta},-q^{-2a+1}e^{-\sqrt{-1}\theta}}{q^2,-q^{-2a+2},-q^{-2a+2}}{q,q}_{.}
\end{align*}
Put $P_n(X)=Q_n(\dfrac{(q+q^{-1}+(q+q^{-1})^{-1}\llbracket a\rrbracket^2)X-(q+q^{-1})^{-1}\llbracket a\rrbracket^2}{2})$. Then $P_0=1$, $\deg{P_n}=n$ and $u_{[a],[a]}^n=P_{n}(u_{[a],[a]}^1)$ for all $n\geq1$.
\end{thm}
\section{$1$-cocycles on coideals}
Let us introduce $1$-cocycles on coideals and consider the condition for them to be extended to $1$-cocycles on Drinfeld doubles.
\begin{defn}[$1$-cocycles on coideals]
Let $\mathcal{A}$ be a $*$-bialgebra and $\mathcal{B}$ be a $*$-subalgebra of $\mathcal{A}$. Suppose $\mathcal{H}$ is a preHilbert space and  $ \pi \colon  \mathcal{B} \to \End(\mathcal{H}) $ is a $*$-representation. Then a linear map $C \colon  \mathcal{B} \to \mathcal{H} $ is called a \emph{$(\pi, \epsilon)$-$1$-cocycle} if $C$ satisfies $C(xy)=\pi(x)C(y)+C(x)\epsilon(y)$ for all $x,y \in \mathcal{B}$. We say that a linear map $C \colon  \mathcal{B}\to \mathcal{H}$ is a \emph{coboundary} if $C$ is of the form $\pi(\bullet)\xi-\epsilon(\bullet)\xi$ for some $\xi \in \mathcal{H}$.
\end{defn}
We define the concept of growth, which describes how $1$-cocycles increase.
\begin{defn}
Let $\mathcal{A}$ be a CQG Hopf $\ast$-algebra and $\mathcal{B}\subset\mathcal{A}$ be a unital left (or right) coideal $\ast$-subalgebra of $\mathcal{A}$. Let $(U_{\alpha}, H_{\alpha})_{\alpha\in I}$ be a maximal family of mutually inequivalent unitary corepresentations of $\mathcal{A}$. Write their matrix coefficients by $u^{\alpha}_{\xi,\eta}$. Then $\mathcal{B}=\bigoplus_{\beta\in J}\spn\{u^{\beta}_{\eta^{\beta}_{i},\eta^{\beta}_{j}}\mid i=1,\ldots,m_{\beta};j=1,\ldots, n_{\beta}\}$, where $J\subset I$ and $(\eta^{\alpha}_i)_{i=1}^{\max{m_{\beta},n_{\alpha}}}$ is an orthonormal system (not necessarily a basis) in $H_{\beta}$ for each $\beta$. Let $C\colon\mathcal{B}\to\mathcal{H}$ be a $1$-cocycle. We put  $C^{\beta}\coloneqq(\id\otimes C)(\sum_{i=1}^{m_{\beta}}\sum_{j=1}^{n_{\beta}}\eta^{\beta}_i\eta^{\beta\ast}_j\otimes u^{\beta}_{\eta^{\beta}_{i},\eta^{\beta}_{j}})$. We say a $1$-cocycle $C$ is \emph{proper} if for all $M>0$ there exists a finite subset $F\subset J$ such that  $C^{\beta\ast}C^{\beta}\geq M\sum_{i=1}^{n_{\beta}}\sum_{j=1}^{n_{\beta}}\eta^{\beta}_i\eta^{\beta\ast}_j$ for all $\beta\in J\setminus{F}$.
\end{defn}
\begin{Rem}
In the case $\mathcal{B}=\mathcal{O}_q(S_t^2)$, we will consider the element $u^{n}_{[a]}\coloneqq\displaystyle\sum_{i=-n}^{n}\eta_{[a+2i]}^{n}\otimes u^{n}_{\eta_{[a+2i]}^{n}, \eta_{[a]}^{n}}$ and omitted $\eta_{[a]}^{\ast}$ since $C^{\beta\ast}C^{\beta}$ is a constant multiple of rank one projection $\eta_{[a]}\eta_{[a]}^{\ast}$. 
\end{Rem}
Let $\pi\colon\mathcal{B}\to\End(\mathcal{H})$ be a $*$-representation of $\mathcal{B}$. Let $C\colon \mathcal{B}\to\mathcal{H}$ be a linear map. Let $\Tilde{\pi}\coloneqq\begin{bmatrix}
\pi & C \\
0 & \epsilon \\
\end{bmatrix}\colon\mathcal{B}\to\End(\mathcal{H}\oplus\mathbb{C})$. We have a well-known characterization:

\begin{prop}\label{char} Let $\mathcal{C}$ be a $\ast$-algebra and $\mathcal{H}$ a preHilbert space.
Let $C \colon  \mathcal{B} \to \mathcal{H}$ be a linear map.
 \begin{enumerate}
\item The map $C$ is a $1$-cocycle if and only if $\Tilde{\pi}$ is a representation.
\item Under the assumption that $C$ is a $1$-cocycle, the map $C$ is a coboundary if and only if $\Tilde{\pi}\simeq \pi \oplus \epsilon$ as representations of $\mathcal{B}$.
\end{enumerate}
\end{prop}

\begin{Rem}
%It is easy to construct coboundaries.
%In the theory of $1$-cocycles for groups, we have $1$-cocycle is bounded if and only if it is a coboundary. In fact, if $\mathcal{A}$ is a CQG algebra and $\mathcal{B}$ is a unital left coideal $*$-subalgebra of $\mathcal{A}$, then $C$ is bounded linear operator if and only if it is a coboundary. Indeed, one can apply the argument in the Appendix C. of \cite{SS}.
For a $1$-cocycle $C$ on a unital left (or right) coideal $\ast$-subalgebra $\mathcal{B}$ of a CQG Hopf $\ast$-algebra $\mathcal{A}$, the $1$-cocycle $C$ is coboundary if and only if it is bounded. See \cite[Theorem 1.2.]{SS}.
\end{Rem}
\begin{defn}[The Yetter--Drinfeld condition]
\label{YD}Let $\mathcal{A}$ be a CQG Hopf $\ast$-algebra. Let $\mathcal{B}\subset\mathcal{A}$ be a unital left coideal $\ast$-subalgebra of $\mathcal{A}$. Let $\mathcal{I}\subset\mathcal{U}$ be the stabilizer dual of $\mathcal{B}$. Let $\mathcal{J}(\subset\mathcal{I}\subset\mathcal{U})$ be a unital right coideal $\ast$-subalgebra of $\mathcal{U}$. We consider the Drinfeld double coideal $\mathcal{B} \bowtie \mathcal{J}$. Let $\pi \colon  \mathcal{B} \bowtie \mathcal{J} \to \End(\mathcal{H})$ be a $*$-representation. We say $(\pi , \epsilon)$-$1$-cocycle $C \colon  \mathcal{B} \to \mathcal{H}$ satisfies the \emph{Yetter--Drinfeld condition} if $C$ satisfies $\pi(\omega)C(x)=C(x \lhd S^{-1}(\omega))$ for all $x \in \mathcal{B}$ and all $ \omega \in \mathcal{J}$.
\end{defn}
Let us prove that Yetter--Drinfeld $1$-cocycles on $\mathcal{B}$ extends to the Drinfeld double coideal $\mathcal{B}\bowtie\mathcal{I}_c$.
\begin{thm}\label{mainext}
Let $\mathcal{A}$ be a CQG Hopf $*$-algebra and $\mathcal{B}$ be a unital left coideal $*$-subalgebra of $\mathcal{A}$. Let $\mathcal{I}_c$ be the stabilizer dual of $\mathcal{B}$. Suppose $(\pi,\mathcal{H})$ is a $*$-representation of $\mathcal{B}\bowtie\mathcal{I}_c$ and $C$ is a $(\pi, \epsilon)$-$1$-cocycle of $\mathcal{B}$. Then the following are equivalent.\\
$(1)$ The cocycle $C$ satisfies the Yetter--Drinfeld condition.\\
$(2)$ The cocycle $C$ extends to a $1$-cocycle $\Tilde{C}$ on the Drinfeld double $\mathcal{B} \bowtie \mathcal{I}_c$ such that %$\Tilde{C}(\omega x)=\pi(\omega)C(x)$ and
$\Tilde{C}(x\omega)=C(x)\epsilon(\omega)$ for all $x \in \mathcal{B} $ and $\omega \in \mathcal{I}_c$. 
\end{thm}
\begin{proof}
$(1)\Rightarrow (2)$: Let $\Tilde{C}$ be a linear map such that %$\Tilde{C}(\omega x)=\pi(\omega)C(x)$ and
$\Tilde{C}(x\omega)=C(x)\epsilon(\omega)$ for $x \in \mathcal{B} $ and $\omega \in \mathcal{I}_c$. %At first we show that $\Tilde{C}$ is well-defined, namely, $\Tilde{C}(\omega x)=\Tilde{C}((x\lhd S^{-1}(\omega_{(2)}))\omega_{(1)})$ and $\Tilde{C}(x \omega)=\Tilde{C}(\omega_{(1)}(x\lhd \omega_{(2)}))$ for all $x \in \mathcal{B}$ and all $ \omega \in \mathcal{I}$. We can show this by direct computations: Let $x \in \mathcal{B} $ and $\omega \in \mathcal{I}$. Then we have $\Tilde{C}(\omega x)=\pi(\omega)C(x)=\Tilde{C}(x\lhd S^{-1}(\omega))=\Tilde{C}((x\lhd S^{-1}(\omega_{(2)})))\epsilon(\omega_{(1)})=\Tilde{C}((x\lhd S^{-1}(\omega_{(2)}))\omega_{(1)})$ and $\Tilde{C}(x\omega)=C(x)\epsilon(\omega)=C(x\lhd \omega_{(2)} \lhd S^{-1}(\omega_{(1)}))=\pi(\omega_{(1)})C(x\lhd \omega_{(2)})=\Tilde{C}(\omega_{(1)}(x \lhd \omega_{(2)}))$. Next, we show that $\Tilde{C}$ is a $1$-cocycle. This also can be proved by direct computations: 
Let $x, y \in \mathcal{B} $ and $\omega, \nu \in \mathcal{I}_c$. Then we have
\begin{align*}
\Tilde{C}(x\omega y\nu)&=\Tilde{C}(x(y\lhd S^{-1}(\omega_{(2)}))\omega_{(1)}\nu)\\
&=C(x(y\lhd S^{-1}(\omega_{(2)}))\epsilon(\omega_{(1)})\epsilon(\nu)\\
&=C(x(y\lhd S^{-1}(\omega)))\epsilon(\nu)\\
&=\pi(x)C(y\lhd S^{-1}(\omega))\epsilon(\nu)+C(x)\epsilon(y)\epsilon(\omega)\epsilon(\nu)\\
&=\pi(x\omega)\Tilde{C}(y\nu)+\Tilde{C}(x\omega)\epsilon(y\nu).
\end{align*}
Hence $\Tilde{C}$ is a $1$-cocycle with the desired property.\\
$(2)\Rightarrow (1)$: Note that $\mathcal{I}_c$ is dense in $\mathcal{I}$ with respect to the topology of pointwise convergence. Define $\Tilde{C}'(x\omega)\coloneqq\lim_{\lambda}\Tilde{C}(x\omega_{\lambda})=\lim_{\lambda}C(x)\epsilon(\omega_{\lambda})$, where $x\in\mathcal{B}$, $\omega\in\mathcal{I}$ and $(\omega_{\lambda})_{\lambda}\subset\mathcal{I}_c$ is a net conveging to $\omega$. Since $\omega_{\lambda}\to\omega$, we have $\epsilon(\omega_{\lambda})=\omega_{\lambda}(1)\to\omega(1)=\epsilon(\omega)$. Thus the limit exists and is independent of the choice of nets. This $\Tilde{C}'$ is an extension of $\Tilde{C}$ and $C$. Let $x \in \mathcal{B}$ and $ \omega \in \mathcal{I}_c$. Then
\begin{align*}
\pi(\omega)C(x)&=\pi(\omega)\Tilde{C}'(x)\\
&=\pi(\omega)\Tilde{C}'(x)+\Tilde{C}'(\omega)\epsilon(x)\\
&=\Tilde{C}'(\omega x)\\
&=\Tilde{C}(\omega x)\\
&=\Tilde{C}((x\lhd S^{-1}(\omega_{(2)}))\omega_{(1)})\\
&=C(x\lhd S^{-1}(\omega_{(2)}))\epsilon(\omega_{(1)})\\
&=C(x\lhd S^{-1}(\omega)).
\end{align*}
This completes the proof.
\end{proof}
\begin{Rem}
Let $C, \Tilde{C}$ and $\Tilde{C}'$ be as in Proposition \ref{mainext}. Then $1$-cocycle $\Tilde{C}$ vanishes on $\mathcal{I}_c$. Indeed, if $\omega\in\mathcal{I}_c$, then
\begin{align*}
    \Tilde{C}(\omega)&=C(1\omega)=C(1)\epsilon(\omega)=0.
\end{align*}
We also have
\begin{align*}
\Tilde{C}(\omega x)&=
\pi(\omega)\Tilde{C}'(x)+\Tilde{C}'(\omega)\epsilon(x)=\pi(\omega)C(x),
\end{align*}
where $x\in\mathcal{B}$ and $\omega\in\mathcal{I}_c$.
\end{Rem}
\begin{prop}\label{iff}
Let $\mathcal{B}$, $\mathcal{I}_c$, $\pi$ and $C$ be as in Proposition \ref{mainext}. Then $C$ is a coboundary if $\Tilde{C}$ is a coboundary.
\end{prop}
\begin{proof}Suppose $\Tilde{C}$ is a coboundary, then there exists some $\xi\in\mathcal{H}$ such that $\Tilde{C}(x\omega)=\pi(x\omega)\xi-\epsilon(x\omega)\xi$ for all $x\in\mathcal{B}$ and $\omega\in\mathcal{I}_c$. We have
\begin{align*}
0=\Tilde{C}(\Phi_{\mathcal{C}})=\pi(\Phi_{\mathcal{C}})\xi-\epsilon(\Phi_{\mathcal{C}})\xi.
\end{align*}
Hence by the fact that $\epsilon(\Phi_{\mathcal{C}})=1$ we have $\pi(\Phi_{\mathcal{C}})\xi=\xi$.
Then, by the Yetter--Drinfeld condition, we have 
\begin{align*}
C(x)=\Tilde{C}(x)\epsilon(\Phi_{\mathcal{C}})=\Tilde{C}(x\Phi_{\mathcal{C}})=\pi(x\Phi_{\mathcal{C}})\xi-\epsilon(x\Phi_{\mathcal{C}})\xi=\pi(x)\xi-\epsilon(x)\xi.
\end{align*}
This means that $C$ is a coboundary.
\end{proof}
%\begin{Rem}
 %   The reverse implication of the proposition above is true at least in the case $\mathcal{B}=\mathcal{O}_q(S_t^2)$ and $(\mathcal{H}, \pi)$ is an $SL(2,\mathbb{R})_t$-admissible $\ast$-representation. Indeed, let $C\colon\mathcal{B}\to\mathcal{H}$ be a coboundary, then there exists a vector $\xi\in\mathcal{H}$ such that $C(\bullet)=\pi(\bullet)-\epsilon(\bullet)$ on $\mathcal{B}$. Extend $C$ to the $1$-cocycle $\Tilde{C}$ on $\mathcal{B}\bowtie\mathcal{I}_c$. We have
  %  \begin{align*}
   %     C(T_a^{\pm})=\Tilde{C}(T_a^{\pm}\Phi_{\mathcal{C}})=0.
    %\end{align*}
    %We also have
    %\begin{align*}
%\pi(A_a)\xi-(q+q^{-1})\xi&=C(A_a)=\Tilde{C}(A_a\Phi_{\mathcal{C}})=\Tilde{C}(\Phi_{\mathcal{C}}A_a)\\
%&=\pi(\Phi_{\mathcal{C}})C(A_a)=\pi(\Phi_{\mathcal{C}})(\pi(A_a)\xi-(q+q^{-1})\xi).
 %   \end{align*}
  %  This means that $C(A_a)=\pi(A_a)\pi(\Phi_{\mathcal{C}})\xi-(q+q^{-1})\pi(\Phi_{\mathcal{C}})\xi$. Hence $C(b)=\pi(b)\xi_0-\epsilon(b)\xi_0$, where $b=T_a^{\pm}$ and $\xi_0\coloneqq\pi(\Phi_{\mathcal{C}})\xi$. Now it follows that $\Tilde{C}(bx)=\pi(bx)\xi_0-\epsilon(bx)\xi_0$, where $b\in\mathcal{B}$ and $x\in\mathcal{I}_c$.
%\end{Rem}
\section{A $1$-cocycle on $SL_q(2,\mathbb{R})$ }
In this section we costruct $1$-cocycles on $SL_q(2,\mathbb{R})$ which is not a coboundary.
\subsection{Decomposing $\mathcal{M}_{q+q^{-1},a}$ into irreducibles}
In order to construct $1$-cocycles, we decompose $\mathcal{M}_{q+q^{-1},a}$ into the direct sum of $\mathcal{D}_{2}^{+}$, $\mathcal{D}_{2}^{-}$ and $\mathds{1}$ as $\mathcal{I}_c$-representations.
\begin{lem}\label{exactseq} We have an exact sequence 
\[0\to   \mathcal{D}_{2}^{+} \oplus \mathcal{D}_{2}^{-} \xrightarrow{\phi}  \mathcal{M}_{q+q^{-1},a}  \xrightarrow{p} 
  \mathds{1} \to  0\]
  in the representation category of $\mathcal{O}_q(S^{2}_t) \bowtie \mathcal{I}_c$, where $\phi$ is the injection $e_{a+2n} \mapsto e_{a+2n} (n\in \mathbb{Z} \setminus{\{0\}})$ and $\pi$ is the canonical projection.
\end{lem}
\begin{proof}
Since $\mathds{1}$ is a $(q+q^{-1}, a)$-basic module, we apply \cite[Proposition 3.11.]{DCDz1} to obtain a surjective intertwiner $p\colon\mathcal{M}_{q+q^{-1},a}\rightarrow\mathds{1}$. We have
\begin{align*}
\mathds{1}=\mathcal{M}_{q+q^{-1},a}/\ker(p).
\end{align*}
We put 
\begin{align*}
\mathcal{K}^{+}&\coloneqq\spn\{e_{a+2n}(\coloneqq T_{a}^{+,(n)}e_a)\in \ker(p)\mid n>0\},\\
\mathcal{K}^{-}&\coloneqq\spn\{e_{a+2n}(\coloneq T_{a}^{-,(-n)}e_a)\in \ker(p)\mid n<0\}.
\end{align*}
Note that the vector spaces $\mathcal{K}^{\pm}$ are submodules. %Indeed, we have
%\begin{align*}
%T^{-}_{a+2}e_{a+2}&=T^{-}_{a+2}T^{+}_{a}e_a\\
%&=(\{a-a+1\}-A_a)(\{a+a+1\}+A_a)e_a\\
%&=(q+q^{-1}-q-q^{-1})(\{2a+1\}+q+q^{-1})e_a\\
%&=0.
%\end{align*}
%In a similar way, we have $T^{+}_{a-2}e_{a-2}=0$. It is clear that $A_{a\pm2}e_{a\pm2}\in\mathcal{K}^{\pm}$ and $T^{\pm}_{a\pm2}e_{a\pm2}\in\mathcal{K}^{\pm}$. Since for any $c\in\mathbb{R}$, the Podle\'{s} sphere $\mathcal{O}_q(S_t^2)$ is generated by the elements $T^{\pm}_{c}$ and $A_{c}$ as an algebra (or we have $\mathcal{O}_q(S_t^2)=\spn\{T_c^{+,(n)}A^k_c,T_c^{-,(n)}A^k_c\mid k,n\in\mathbb{Z}_{\geq0}\}$), we have $xe_{a\pm2n}\in\mathcal{K}^{\pm}$  for any $x\in\mathcal{O}_q(S_t^2)$. It follows that $\mathcal{K}^{\pm}$ are submodules.
Indeed, it suffices to show that $u^1_{\xi,\eta^n_{[a]}}e_{a+2n}\in\mathcal{K}^+$ for all $\xi\in H_1$ and $n>0$. The matrix coefficient $u^1_{\xi,\eta^n_{[a]}}$ is of the form
\[u^1_{\xi,\eta^n_{[a]}}=\alpha T^{+}_{a+2n}+\beta A_{a+2n}+\gamma T^{-}_{a+2n}\] for some $\alpha, \beta$ and $\gamma\in\mathbb{C}$. It is clear that $T^{+}_{a+2n}e_{a+2n}, A_{a+2n}e_{a+2n}\in\mathcal{K}^{+}$ and it follows that
\begin{align*}
    T^{-}_{a+2n}e_{a+2n}&=(\{2n-1\}-(q+q^{-1}))(\{2a+2n-1\}+q+q^{-1})e_{a+2n-2}\\
    &\left\{
\begin{array}{ll}
\in\mathcal{K}^{+} & (n>1) \\
=0 & (n=1)
\end{array}
\right._{.}
\end{align*}
Thus $u^1_{\xi,\eta^n_{[a]}}e_{a+2n}\in\mathcal{K}^+$ and so $\mathcal{K}^{+}$ is a submodule. In a similar way $\mathcal{K}^{-}$ is a submodule. 
We have, by construction, an intertwiner $Q\colon \mathcal{M}_{q+q^{-1}, a+2}\twoheadrightarrow \mathcal{L}_{q+q^{-1}, a+2}=\mathcal{D}_{2}^{+}$. Since $\mathcal{K}^{+}$ is $(q+q^{-1}, a+2)$-basic module, there exists an intertwiner (by \cite[Proposition 3.11.]{DCDz1} again) $Q^{+}\colon\mathcal{M}_{q+q^{-1}, a+2}\twoheadrightarrow  \mathcal{K}^{+}$. Since $\ker(Q)=\spn\{e_{a+2n}\mid n\leq0\}$, any $\xi\in\ker(Q)$ is of the form $\xi=\sum_{n<0}a_{n}e_{a+2n}$ and we have $Q^{+}(\xi)=0$. Then there exists an intertwiner $\pi^{+}\colon\mathcal{L}_{q+q^{-1}, a+2}=\mathcal{D}_{2}^{+}\twoheadrightarrow\mathcal{K}^{+}$. In summary, we have the following commutative diagram.
\[
\xymatrix{
\mathcal{D}_{2}^{+}\ar@{.>}[rd]^-{\pi^{+}}\\
\mathcal{M}_{q+q^{-1}, a+2}\ar[r]^-{Q^{+}}\ar[u]^-{Q}&\mathcal{K}^{+}\\
\ker(Q)\ar[u]
}
\]
Since $\mathcal{D}_{2}^{+}$ is irreducible, the map $\pi^{+}$ is injective. Hence we may consider that $\mathcal{D}_{2}^{+}\subset \mathcal{K}^{+}$ as representations of $\mathcal{O}_q(S^{2}_t) \bowtie \mathcal{I}_c$. In a similar way, we may consider that $\mathcal{D}_{2}^{-}\subset \mathcal{K}^{-}$. Hence $\mathcal{D}_{2}^{+}\oplus\mathcal{D}_{2}^{-}\subset \mathcal{K}^{+}\oplus \mathcal{K}^{-}$. By considering the weights, it follows that $\mathcal{D}_{2}^{+}\oplus\mathcal{D}_{2}^{-}=\mathcal{K}^{+}\oplus \mathcal{K}^{-}=\ker{p}$. This proves that $\phi$ is an injective intertwiner which sends $e_{a+2n}$ to $e_{a+2n} (n\in \mathbb{Z} \setminus{\{0\}})$ and  $\Ima{\phi}=\ker{p}$.
%Next we show the exactness. $\pi \circ \phi=0$ is trivial by definition. Suppose $\pi(x)=0$ and write $x=\sum x_{n}e_{a+2n}$. Then $x_{0}=0$. This shows that $x\in \mathcal{D}_{2}^{+} \oplus \mathcal{D}_{2}^{-}$.
\end{proof}

In the category of vector spaces, the exact sequence above splits. In fact, there exists an isomorphism  $\mathcal{M}_{q+q^{-1},a} \simeq (\mathcal{D}_{2}^{+} \oplus \mathcal{D}_{2}^{-}) \oplus \mathds{1}$ as $\mathcal{I}_c$-representations.
However, the sequence does not split in the category of $ \mathcal{O}_q(S^{2}_t) \bowtie \mathcal{I}_c$. As a final remark in this subsection, we note that $\mathcal{M}_{q+q^{-1},a}$ is not irreducible since $\mathcal{M}_{q+q^{-1},a}$ contains an invariant sub space $\mathcal{D}_{2}^{+} \oplus \mathcal{D}_{2}^{-}$, which contains no nonzero vectors of weight $[a]$.

\subsection{An invariant inner product on $\mathcal{D}_{2}^{+} \oplus \mathcal{D}_{2}^{-}$}
\label{inv}
Let $0<\lambda<q+q^{-1}$ and let $\langle \bullet, \bullet \rangle$ be the inner product on $\mathcal{M}_{\lambda, a}=\mathcal{L}_{\lambda, a}$ such that $\langle e_{a}, e_{a} \rangle=1$. We define \begin{align*}\langle \bullet, \bullet \rangle_{_{_{\lambda}}}\coloneqq \{a+2\}\{a\}^{-1}(q+q^{-1}-\lambda)^{-1}(\{2a+1\}+\lambda)^{-1}\langle \bullet, \bullet \rangle.
\end{align*}
Then it follows that \begin{align*}
&\langle T_{a}^{+}e_{a}, T_{a}^{+}e_{a}\rangle_{_{_{\lambda}}}\\
&=\{a+2\}\{a\}^{-1}(\{1\}-\lambda)^{-1}(\{2a+1\}+\lambda)^{-1}\langle T_{a}^{+}e_{a}, T_{a}^{+}e_{a}\rangle\\
&=\{a+2\}\{a\}^{-1}(\{1\}-\lambda)^{-1}(\{2a+1\}+\lambda)^{-1}\langle T_{a}^{-}T_{a}^{+}e_{a}, e_{a}\rangle\\
&=\{a+2\}\{a\}^{-1}(\{1\}-\lambda)^{-1}(\{2a+1\}+\lambda)^{-1}\{a\}\{a+2\}^{-1}\langle T_{a+2}^{-}T_{a}^{+}e_{a}, e_{a} \rangle\\
&=(\{1\}-\lambda)^{-1}(\{2a+1\}+\lambda)^{-1} \langle (\{a-a+1\}-A_{a})(\{a+a+1\}+A_{a})e_{a}, e_{a} \rangle\\
&=(\{1\}-\lambda)^{-1}(\{2a+1\}+\lambda)^{-1} \langle (\{a-a+1\}-\lambda)(\{a+a+1\}+\lambda)e_{a}, e_{a} \rangle\\
&=1.
\end{align*}
Hence  $\langle \bullet, \bullet \rangle_{_{_{\lambda}}}$ is an invariant inner product on $\mathcal{M}_{\lambda, a}$ such that $\langle T_{a}^{+}e_{a}, T_{a}^{+}e_{a}\rangle_{_{_{\lambda}}}=1$. We compute the limit of this map:
\begin{prop}
\label{inn}
We have\\
\[
\lim_{\lambda \to q+q^{-1}}  \langle \xi, \eta \rangle_{_{_{\lambda}}}=
 \begin{cases}
    \langle \xi, \eta \rangle_{_{_{\mathcal{D}_{2}^{+}}}} & \text{if $\xi, \eta \in \mathcal{D}_{2}^{+} $,} \\
    \infty  & \text{if $\xi, \eta \in \mathcal{N}\setminus\{0\} $,} \\
    \widetilde{\langle \xi, \eta \rangle}_{_{_{\mathcal{D}_{2}^{-}}}} & \text{if $\xi, \eta \in \mathcal{D}_{2}^{-} $,}
  \end{cases}
\]
where $\mathcal{N}$ is the subspace of $\mathcal{M}_{q+q^{-1}, a}$ consisting of the vectors with weight $ [a] $ and $\widetilde{\langle \xi, \eta \rangle}_{_{_{\mathcal{D}_{2}^{-}}}}\coloneqq \{a+2\}\{a-1\}\{a-2\}^{-1}\{a+1\}^{-1}\langle \xi, \eta \rangle_{_{_{\mathcal{D}_{2}^{-}}}}$.
\end{prop}
\begin{proof}The computations in this proof are essentially the same as in \cite[Proposition 3.14.]{DCDz1}.\\
$\bullet$ Let $\xi$ and $\eta \in \mathcal{D}_{2}^{+}$. Since $\langle T_{a}^{+,(n)}e_a, T_{a}^{+,(m)}e_a \rangle_{_{_{\lambda}}}=0$ if $n\neq m$, we may assume that \begin{align*}
\xi=\eta=T_{a}^{+,(n)}e_a=T_{a+2n-2}^{+}T_{a+2n-4}^{+}\cdots T_{a}^{+}e_{a}. 
\end{align*}
Then we have the following computation:\\
\\
$\langle \xi, \eta \rangle_{_{_{\lambda}}}\\
=\langle T_{a+2n-2}^{+}T_{a+2n-4}^{+}\cdots T_{a}^{+}e_{a}, T_{a+2n-2}^{+}T_{a+2n-4}^{+}\cdots T_{a}^{+}e_{a}\rangle_{_{_{\lambda}}}\\
=\langle T_{a+2n-2}^{-}T_{a+2n-2}^{+}T_{a+2n-4}^{+}\cdots T_{a}^{+}e_{a}, T_{a+2n-4}^{+}T_{a+2n-6}^{+}\cdots T_{a}^{+}e_{a}\rangle_{_{_{\lambda}}}\\
=\dfrac{\{a+2n-2\}}{\{a+2n\}}\langle T_{a+2n}^{-}T_{a+2n-2}^{+}T_{a+2n-4}^{+}\cdots T_{a}^{+}e_{a}, T_{a+2n-4}^{+}T_{a+2n-6}^{+}\cdots T_{a}^{+}e_{a}\rangle_{_{_{\lambda}}}\\
=\dfrac{\{a+2n-2\}}{\{a+2n\}}\langle (\{2n-1\}-A_{a+2n-2})(\{2n-1+2a\}+A_{a+2n-2})T_{a+2n-4}^{+}\cdots T_{a}^{+}e_{a}, \\
T_{a+2n-4}^{+}T_{a+2n-6}^{+}\cdots T_{a}^{+}e_{a}\rangle_{_{_{\lambda}}}\\
=\dfrac{\{a+2n-2\}}{\{a+2n\}}(\{2n-1\}-\lambda)(\{2n-1+2a\}+\lambda)\langle T_{a+2n-4}^{+}\cdots T_{a}^{+}e_{a}, \\
T_{a+2n-4}^{+}T_{a+2n-6}^{+}\cdots T_{a}^{+}e_{a}\rangle_{_{_{\lambda}}}\\
\vdots\\
=\dfrac{\{a+2\}}{\{a+2n\}}(\displaystyle \prod_{k=1}^{n-1}(\{2k+1\}-\lambda))(\displaystyle \prod_{k=1}^{n-1}(\{2k+1+2a\}-\lambda))\langle T_{a}^{+}e_{a}, T_{a}^{+}e_{a} \rangle_{_{_{\lambda}}}$\\
$\to\dfrac{\{a+2\}}{\{a+2n\}}(\displaystyle \prod_{k=1}^{n-1}(\{2k+1\}-q-q^{-1}))(\displaystyle \prod_{k=1}^{n-1}(\{2k+1+2a\}-q-q^{-1}))$ as $\lambda \to q+q^{-1}$.

On the other hand, the similar computation shows that 
\begin{align*}
&\langle \xi, \eta \rangle_{_{_{\mathcal{D}_{2}^{+}}}}=\dfrac{\{a+2\}}{\{a+2n\}}(\displaystyle \prod_{k=1}^{n-1}(\{2k+1\}-q-q^{-1}))(\displaystyle \prod_{k=1}^{n-1}(\{2k+1+2a\}-q-q^{-1}))
\end{align*}
$\bullet$ Let $\xi$ and $\eta$ $\in \mathcal{N}$. We may assume that $\xi=\eta=e_{a}$. Then $\langle \xi, \eta \rangle_{_{_{\lambda}}}=\{a+2\}\{a\}^{-1}(q+q^{-1}-\lambda)^{-1}(\{2a+1\}+\lambda)^{-1}
\to \infty$ as $\lambda \to q+q^{-1}$.\\
$\bullet$ Let $\xi$ and $\eta \in \mathcal{D}_{2}^{-}$. As in the former cases, we may assume that
\begin{align*}
\xi=\eta=T_{a}^{-,(n)}e_a=
T_{a-2n+2}^{-}T_{a-2n+4}^{-}\cdots T_{a}^{-}e_{a}.
\end{align*}
Then by a similar computation as above we have
\begin{align*}
\displaystyle\lim_{\lambda }  \langle \xi, \eta \rangle_{_{_{\lambda}}}=\dfrac{\{a+2\}}{\{a-2\}}\dfrac{\{a-1\}}{\{a+1\}}\langle \xi, \eta \rangle_{_{_{\mathcal{D}_{2}^{-}}}}.
\end{align*}
\end{proof}
Restricting the map in Proposition \ref{inn}, we get an invariant inner product on the space $\mathcal{D}_{2}^{+} \oplus \mathcal{D}_{2}^{-}$. In the sequel, we equip $\mathcal{D}_{2}^{+} \oplus \mathcal{D}_{2}^{-}$ with this inner product. 
\subsection{A construction of $1$-cocycles on $SL_q(2,\mathbb{R})$ from a representation }
 In what follows, let $\pi \colon \mathcal{O}_q(S^{2}_t) \bowtie \mathcal{I}_c \to \End(\mathcal{D}_{2}^{+} \oplus \mathcal{D}_{2}^{-})$ denote the $*$-representation of $\mathcal{O}_q(S^{2}_t) \bowtie \mathcal{I}_c$ on $\mathcal{D}_{2}^{+} \oplus \mathcal{D}_{2}^{-}$ and $\Tilde{\pi}$ denote the linear map $\Tilde{\pi}\colon\mathcal{B}\to\End(\mathcal{M}_{q+q^{-1},a})$. By Proposition \ref{char}, Theorem \ref{mainext} and the fact that $\Tilde{\pi}$ is a representation, the following theorem holds.
\begin{thm}\label{construction1}
For $b\in\mathcal{O}_q(S^{2}_t)$ and $x\in\mathcal{I}_c$, we put $\Tilde{C}(bx)=\Tilde{\pi}(bx)e_{a}-\epsilon(bx)e_{a}$ and $C(b)=\Tilde{\pi}(b)e_a-\epsilon(b)e_a$. Then $\Tilde{C}$ is a $(\pi, \epsilon)$-$1$-cocycle and $C$ is a Yetter--Drinfeld $(\pi, \epsilon)$-$1$-cocycle.
\end{thm}

\begin{Rem}
Our $1$-cocycles are not coboundaries. This fact follows from Proposition \ref{iff} and the remark after Lemma \ref{exactseq}. This also follows from Theorem \ref{proper}.
\end{Rem}
\section{Some properties of our $1$-cocycles}
In this section, we consider some properties our $1$-cocycles have. We see the following properties:
\begin{enumerate}
    \item How our $1$-cocycles $C$ and $\Tilde{C}$ decompose into `irreducibles'.
    \item Our $1$-cocycle $C$ exhausts all (non-trivial) Yetter--Drinfeld $1$-cocycles.
    \item Our $1$-cocycle $C$ is proper.
\end{enumerate}

\subsection{A decomposition}\label{decomp}
By decomposing $\mathcal{D}_{2}^{+} \oplus \mathcal{D}_{2}^{-}$, our $1$-cocycle $\Tilde{C}$ is written as the sum of two $1$-cocyles. Let $H$ be the completion of  $\mathcal{D}_{2}^{+} \oplus \mathcal{D}_{2}^{-}$. Let ${D}_{2}^{+}$ and ${D}_{2}^{-}$ be the completions of the spaces $\mathcal{D}_{2}^{+}$ and $\mathcal{D}_{2}^{-}$ respectively. Suppose $P_{+}$ and $P_{-}$ be the orthogonal projections onto ${D}_{2}^{+}$ and ${D}_{2}^{-}$. Let $\pi_{\pm}$ denotes the representations on $\mathcal{D}_{2}^{\pm}$ respectively. Then we have
\begin{prop}
$\Tilde{C}_{+}\coloneqq P_{+}\Tilde{C}$ and $\Tilde{C}_{-}\coloneqq P_{-}\Tilde{C}$ are $1$-cocycles with respect to $(\pi^{+}_{2}, \epsilon)$ and $(\pi^{-}_{2}, \epsilon)$ respectively.
\end{prop}
\begin{proof}
We can show this proposition by a direct computation. let $b,c\in\mathcal{O}_{q}(S_t^2)\bowtie\mathcal{I}_c$, then we have
\begin{align*}
&P_{+}\Tilde{C}(bc)\\
&=P_{+}(\pi(b)\Tilde{C}(c)+\Tilde{C}(b)\epsilon(c))\\
&=P_{+}(\pi_{+}\oplus\pi_{-})(b)(P_{+}\Tilde{C}(c)+P_{-}\Tilde{C}(c))+P_{+}\Tilde{C}(b)\epsilon(c)\\
&=\pi_{+}(b)P_{+}\Tilde{C}(c)+P_{+}\Tilde{C}(b)\epsilon(c).
\end{align*}
In a similar way, we can prove $P_{-}\Tilde{C}(bc)=\pi_{-}(b)P_{-}\Tilde{C}(c)+P_{-}\Tilde{C}(b)\epsilon(c).$
\end{proof}
We define $1$-cocycles $C_{\pm}\coloneqq P_{\pm}C$. One can show that $\Tilde{C}_{\pm}(bx)=C_{\pm}(b)\epsilon(x)\hspace{1mm}(b\in\mathcal{O}_q(S_t^2), x\in\mathcal{I}_c)$ and then $C_{\pm}$ satisfy the Yetter--Drinfeld condition.
\subsection{A classification} 
We have constructed examples of Yetter--Drinfeld $1$-cocycles on $\mathcal{O}_q(S_t^2)$. Then one may want to find all of Yetter--Drinfeld $1$-cocycles on $\mathcal{O}_q(S_t^2)$. Since it is easy to find all of coboundaries, let us find all of Yetter--Drinfeld $(\pi', \epsilon)$-$1$-cocycles, which is not a coboundary, where $\pi'=(\pi', \mathcal{K})$ is an %irreducible
$SL(2,\mathbb{R})_t$-admissible $*$-representation.\par
Suppose $C'\colon\mathcal{O}_q(S_t^2)\to\mathcal{K}$ is a surjective Yetter--Drinfeld $(\pi', \epsilon)$-$1$-cocycle which is not a coboundary. Then we have 
\begin{align*}
    \mathcal{K}=\{C'(x)\mid x\in\mathcal{O}_q(S_t^2)\}.
\end{align*}
Let $\mathcal{L}=\mathcal{K}\oplus \mathds{1}$ (the direct sum of vector spaces). Extend $C'$ to the $1$-cocycle $\Tilde{C'}\colon\mathcal{O}_q(S_t^2)\bowtie\mathcal{I}\to\mathcal{K}$. Since $\Tilde{C'}$ is a $1$-cocycle, the linear map $\Tilde{\pi}'\coloneqq\begin{bmatrix}
\pi' & \Tilde{C'} \\
0 & \epsilon \\
\end{bmatrix}$
is a representation of $\mathcal{O}_q(S^{2}_t)\bowtie\mathcal{I}$ on $\mathcal{L}$. %Since $C'$ is not a coboundary, we have $\Tilde{\pi}'\ncong\pi'\oplus \epsilon$ as representations of $\mathcal{O}_q(S^{2}_t)\bowtie\mathcal{I}_c$. 
Since $C'$ is surjective, the vector $e_a\in\mathds{1}$ is a cyclic vector for the representation $(\mathcal{L}, \Tilde{\pi}')$. We have 
\begin{align*}
\Tilde{\pi}'(A_a)e_a&=\begin{bmatrix}
\pi'(A_a) & C'(A_a) \\
0 & \epsilon(A_a) \\
\end{bmatrix}   
\begin{bmatrix}
0 \\
1 \\
\end{bmatrix}\\
&=\begin{bmatrix}
\pi'(A_a) & 0 \\
0 & q+q^{-1} \\
\end{bmatrix}
\begin{bmatrix}
0 \\
1 \\
\end{bmatrix}\\
&=
\begin{bmatrix}
0 \\
q+q^{-1} \\
\end{bmatrix}\\
&=(q+q^{-1})e_a.
\end{align*}
Thus the representation $\mathcal{L}$ is  $(q+q^{-1},a)$-basic.
   Then by the universality of basic modules (Proposition \ref{univ}) there exists a surjective intertwiner $s\colon\mathcal{M}_{q+q^{-1},a}\twoheadrightarrow\mathcal{L}$ such that the right half side of the following diagram commutes:
\[
  \xymatrix{
    0\ar[r]& \mathcal{D}_{2}^{+}\oplus\mathcal{D}_{2}^{+} \ar[r]^{\phi} \ar[d]_{t} &  \mathcal{M}_{q+q^{-1},a} \ar[r]^{p} \ar[d]^{s} &  \mathds{1} \ar[r] \ar@{=}[d] &0 \quad\text{(exact)}\\
   0\ar[r] & \mathcal{K} \ar[r] & \mathcal{L} \ar[r] & \mathds{1}\ar[r] &0\quad\text{(exact)}.
  }
\]
By the universality of kernels, there exists an intertwiner $t\colon\mathcal{D}_{2}^{+}\oplus\mathcal{D}_{2}^{-}\rightarrow\mathcal{K}$ such that the diagram above commutes. Since $s$ is surjective, $t$ is also surjective. This implies that $t=\alpha P_{+}+\beta P_{-}$ for some constants $\alpha$ and $\beta$. 
%is a constant multiple of the projection $P_{+}$ such that the  diagram above commutes.
%Thus $C'$ is a constant multiple of $P_{+}\Tilde{C}$. For the latter extension, one can show that $C'$ is a constant multiple of $P_{+}\Tilde{C}$ in a similar way. By a similar argument, we can show that any $(\pi',\epsilon)$-$1$-cocycle $C'$, where $\pi'$ is not necessarily irreducible, 
Therefore $C'$ is of the form $\alpha C_{+}+\beta C_{-}$ and $\restr{\Tilde{C'}}{\mathcal{O}_q(S_t^2)\bowtie\mathcal{I}_c}$ is of the form $\alpha \Tilde{C}_{+}+\beta \Tilde{C}_{-}$.
   %Then $C'\coloneqq C_++C_-$ is a $(\pi_{\mathcal{D}_{2}^{+} \oplus \mathcal{D}_{2}^{-}}, \epsilon)$-$1$-cocycle corresponding to the short exact sequence
   %\[0\xrightarrow & \mathcal{D}_{2}^{+} \oplus \mathcal{D}_{2}^{-}  \xrightarrow & \mathcal{L}^+\oplus \mathcal{L}^- & \xrightarrow &
  %\mathds{1} \xrightarrow & 0,\]
  %and $C_{\pm}=P_{\pm}C'$, where $P_{\pm}$ are orthogonal projections appeared in subsection \ref{decomp}. By the universality of basic modules(\cite[Proposition 3.11.]{DCDz1}), there exists a surjective intertwiner $M_{q+q^{-1}, a}	\twoheadrightarrow & \mathcal{L}^+\oplus \mathcal{L}^-$ . Comparing the weights, one can see that this surjective arrow must be injective. Therefore $C'=\Tilde{C}$ and $C_{\pm}=P_{\pm}\Tilde{C}$.
\subsection{The growth of $C$}
In this subsection, we consider the growth of our $1$-cocycle $C$ by considering the orthogonal polynomials $(P_n)_n$, which appear in Subsection \ref{AW}.
Let $u^{n}_{[a]}\coloneqq\displaystyle\sum_{i=-n}^{n}\eta_{[a+2i]}^{n}\otimes u^{n}_{\eta_{[a+2i]}^{n}, \eta_{[a]}^{n}}$. We define $C^n\coloneqq(\id\otimes C)u^{[n]}$ (cf.\ \cite[Subsection 6.2.]{DFSW}). We call $C^{n*}C^{n}$ the \emph{growth} of $C$. Put $\omega_{\lambda}(\bullet)=\langle e_a, \pi_{\lambda}(\bullet)e_a\rangle$, where $\pi_{\lambda}=(\mathcal{L}_{\lambda, a}, \pi_{\lambda})$. We have
\begin{align*}
&C^{n*}C^{n}\\
&=\displaystyle\lim_{\lambda \to \{1\}}\dfrac{\{a+2\}}{\{a\}(q+q^{-1}-\lambda)(\{2a+1\}+\lambda)}\displaystyle\sum_{i=-n}^{n}\langle \pi_{\lambda}(u^{n}_{\eta_{[a+2i]}^{n},\eta_{[a]}^{n}})e_{a}-\epsilon(u^{n}_{\eta_{[a+2i]}^{n}, \eta_{[a]}^{n}})e_{a}, \\
&\pi_{\lambda}(u^{n}_{\eta_{[a+2i]}^{n}, \eta_{[a]}^{n}})e_{a}-\epsilon(u^{n}_{\eta_{[a+2i]}^{n}, \eta_{[a]}^{n}})e_{a} \rangle\\
&=\dfrac{\{a+2\}}{\{a\}(\{2a+1\}+q+q^{-1})}\displaystyle\lim_{\lambda \to \{1\}}\sum_{i=-n}^{n}(q+q^{-1}-\lambda)^{-1}(\omega_{\lambda}((u^{n}_{\eta_{[a+2i]}^{n}, \eta_{[a]}^{n}})^{*}(u^{n}_{\eta_{[a+2i]}^{n}, \eta_{[a]}^{n}}))\\
&-2\operatorname{Re}\epsilon((u^{n}_{\eta_{[a+2i]}^{n}, \eta_{[a]}^{n}})^{*})\omega_{\lambda}((u^{n}_{\eta_{[a+2i]}^{n}, \eta_{[a]}^{n}}))+
\epsilon((u^{n}_{\eta_{[a+2i]}^{n}, \eta_{[a]}^{n}})^{*}(u^{n}_{\eta_{[a+2i]}^{n}, \eta_{[a]}^{n}})))
\end{align*}
\begin{align*}
&=\dfrac{\{a+2\}}{\{a\}(\{2a+1\}+q+q^{-1})}\displaystyle\lim_{\lambda \to \{1\}}(q+q^{-1}-\lambda)^{-1}(2-2\operatorname{Re}\epsilon((u^{n}_{\eta_{[a]}^{n}, \eta_{[a]}^{n}})^{*})\omega_{\lambda}((u^{n}_{\eta_{[a]}^{n}, \eta_{[a]}^{n}}))\\
%\end{align*}
%\begin{align*}
&=\dfrac{2\{a+2\}}{\{a\}^{2}\{a+1\}}\displaystyle\lim_{\lambda \to \{1\}}\dfrac{1-\operatorname{Re}(P_{n}(\tfrac{\lambda+(q+q^{-1})^{-1}\llbracket a\rrbracket^{2}}{q+q^{-1}+(q+q^{-1})^{-1}\llbracket a\rrbracket^2}))}{q+q^{-1}-\lambda}\\
&=\dfrac{2\{a+2\}}{\{a\}^{2}\{a+1\}}\operatorname{Re}(P'_{n}(\tfrac{q+q^{-1}+(q+q^{-1})^{-1}\llbracket a\rrbracket^{2}}{q+q^{-1}+(q+q^{-1})^{-1}\llbracket a\rrbracket^2})).
\end{align*}
Hence we have proved the following theorem:
\begin{thm}\label{growth}
$C^{n*}C^{n}\asymp P'_{n}(1)\asymp Q'_{n}((q+q^{-1})/2)$.
In other words, the growth of our $1$-cocycle $C$ is equal to that of $P'_{n}(1)$ and $Q'_n((q+q^{-1})/2)$.
\end{thm}
One can prove that $C$ is proper. To prove this, we need the following lemma.
\begin{lem}\label{Arano-sensei}
Let $a,b$ be real numbers with $a<b$. Let $\mu$ be a probability measure on the interval $[a,b]$. Assume that $\mu([b-\epsilon,b])>0$ and let $(P_n)_n$ be a sequence of orthogonal polynomials with respect to $\mu$. Suppose that the $P_n$'s are normalized so that $P_n(b)=1$. Then it follows that
\begin{align*}
    \displaystyle\lim_{n\to\infty}P'_n(b)=\infty.
\end{align*}
\end{lem}
\begin{proof}
Notice that 
\begin{enumerate}
    \item For each $n$, the polynomial $P_n$ has distinct simple zeros in $(a,b)$ (\cite[Theorem 3.3.1.]{Szg})
    \item For all $\epsilon>0$, there exists some sufficiently large number $N>0$ such that $P_n$ has a zero in the interval $[b-\epsilon,b]$ for all $n\geq N$ (\cite[Theorem 6.1.1.]{Szg}).
\end{enumerate}
These imply that the largest zero of $P_n$ converges to $b$. We write zeros of $P_n$ by
\begin{align*}
    a<x_1^n<x_2^n<\cdots<x_n^n<b.
\end{align*}
For each $k=1,\ldots,n-1$, there exists zero $y^n_k$ of $P'_n$ in the interval $(x_k^n,x_{k+1}^n)$. Then for each $k=1,\ldots,n-2$, there exists a zero $z_k^n$ of $P''_n$ in the interval $(y_k^n,y_{k+1}^n)$. Hence $P''_n>0$ on the interval $(z^n_{n-2},\infty)$. This implies that $P'_n$ is increasing on the interval $(y_{n-1}^n,\infty)$. Therefore $P_n$ is convex on the interval $(x_n^n,\infty)$. Then we have
\begin{align*}
    P'_n(b)&>\dfrac{P_n(b)-P_n(x_n^n)}{b-x_n^n}\\
    &=\dfrac{1}{b-x_n^n}\\
    &\to\infty.
\end{align*}
as $n\to\infty$. 
\end{proof}
Let us find the distribution $\mu$ of our polynomial $P_n$ (or $Q_n)$ which appeared in Subsection \ref{AW}. By the orthogonality relations, the $u^n_{[a],[a]}$'s, which are polynomials in $u^1_{[a],[a]}$, are orthogonal with respect to the Haar state $\Phi$ of $SU_q(2)$. Note that $\mathbb{C}[u^1_{[a],[a]}]\subset C^{\ast}\{1,u^1_{[a],[a]}\}\simeq C(\sigma(u^1_{[a],[a]}))$. By the Riesz--Markov--Kakutani theorem, any state on the algebra $C(\sigma(u^1_{[a],[a]}))$ arises as an integration with respect to a probability measure over $\sigma(u^1_{[a],[a]})$. Hence the restriction of the Haar state $\restr{\Phi}{\mathbb{C}[u^1_{[a],[a]}]}$ arises as an integration with respect to some probability measure $\mu$. It follows from the faithfulness of the Haar state $\Phi$ on $\mathcal{O}_q(SU(2))$ that $\supp(\mu)=\sigma(u^{1}_{[a],[a]})$. Then for all $\epsilon>0$, we have $\mu((1-\epsilon,1])\neq0$. Then we get:
\begin{thm}\label{proper}
Our $1$-cocycle $C$ is proper. In particular, $C$ is not a coboundary.
\end{thm}
\begin{proof}
Apply Lemma \ref{Arano-sensei}.
\end{proof}

\end{document}